\documentclass{amsart}
\usepackage[utf8]{inputenc}
\usepackage{amsfonts,amssymb,amsmath,amstext,amscd,epsfig}
\usepackage{amsthm,cases,color,graphicx,graphics}
\usepackage{stackengine}
\usepackage{mathtools}
\usepackage{mathrsfs}
\usepackage{dsfont}
\usepackage[usenames,dvipsnames,svgnames,table]{xcolor}
\usepackage[linktocpage=true,colorlinks=true,linkcolor=RubineRed!85!black,citecolor=ForestGreen,urlcolor=green,pdfborder={0 0 0}]{hyperref}

\numberwithin{equation}{section}
\newtheorem{theorem}{Theorem}[section]
\newtheorem{proposition}[theorem]{Proposition}
\newtheorem{corollary}[theorem]{Corollary}
\newtheorem{definition}[theorem]{Definition}
\newtheorem{lemma}[theorem]{Lemma}

\newcommand{\C}{\mathbb{C}}

\newcommand{\R}{\mathbb{R}}
\newcommand{\T}{\mathbb{T}}
\newcommand{\N}{\mathbb{N}}
\newcommand{\HH}{\mathcal{H}}

\newcommand{\ee}{{\boldsymbol{e}}}

\title[Propagation of smallness in the plane]{Propagation of smallness for solutions of elliptic equations in the plane}

\author{Yuzhe Zhu} 
\address{Department of Pure Mathematics and Mathematical Statistics, University of Cambridge, Wilberforce Road, Cambridge CB3 0WA, UK}
\curraddr{Department of Mathematics, University of Chicago, Chicago, Illinois 60637, USA}
\email{yuzhezhu@uchicago.edu}

\begin{document}
\maketitle
	
\begin{abstract}
We explore quantitative propagation of smallness for solutions of two-dimensional elliptic equations from sets of positive $\delta$-dimensional Hausdorff content for any $\delta>0$. In particular, the gradients of solutions to divergence form equations with H\"older continuous coefficients, as well as those of nondivergence form equations with measurable coefficients, can be quantitatively estimated from the small sets. 
\end{abstract}

\section{Introduction}\label{introduction}
Let $B_R$ be the Euclidean ball in $\R^2$ centred at the origin with radius $R>0$. We aim in this work at studying quantitative propagation of smallness for solutions to the two-dimensional elliptic equation in divergence form
\begin{align}\label{eq1}
\nabla\cdot(A(z) \nabla u(z))=0 {\quad\rm in\ } B_4,  
\end{align}
or solutions to the equation in nondivergence form	\begin{align}\label{eqnon}
\sum\nolimits_{j,k=1}^2 a_{jk}(z)\partial_{x_j}\partial_{x_k}u(z)= 0 {\quad\rm in\ } B_4. 
\end{align}
Here $z=(x_1,x_2)\in\R^2$, and the real symmetric matrix $A(z)=(a_{jk}(z))_{2\times 2}$ is elliptic, that is, there is some constant $\Lambda>1$ such that 
\begin{align}\label{elliptic}
\Lambda^{-1}|\xi|^2 \le \sum\nolimits_{j,k=1}^2a_{jk}(z)\xi_j\xi_k \le \Lambda|\xi|^2 {\quad\rm for\ any\ } \xi\in\R^2,\ z\in B_4. 
\end{align}
When focusing on the properties of the gradient $\nabla u$ for $u$ solving \eqref{eq1}, we have to suppose that the leading coefficients are H\"older continuous, that is, there is some constant $\gamma\in(0,1]$ such that 
\begin{align}\label{Holder}
|a_{jk}(z)-a_{jk}(z')|\le \Lambda |z-z'|^\gamma {\quad\rm for\ any\ } j,k\in\{1,2\},\ z,z'\in B_4.  
\end{align}
	
We recall that solutions to \eqref{eq1} are those functions lying $H^1_{loc}(B_4)$ satisfying \eqref{eq1} in the sense of distributions, and solutions to \eqref{eqnon} are those functions lying $H^2_{loc}(B_4)$ satisfying \eqref{eqnon} almost everywhere. 
	
The goal of the present note is to show the propagation of smallness for solutions from any $\omega\subset B_1$ lying on a line with $\HH_\delta(\omega)>0$, and propagation of smallness for gradients from any $\Omega\subset B_1$ with $\HH_\delta(\Omega)>0$, for any fixed $\delta>0$. Here we denote by $\HH_\delta$ the $\delta$-dimensional Hausdorff content, that is, for a subset $E\subset\R^2$, $\HH_\delta(E):=\inf\big\{\sum\nolimits_{j\in\N}r_j^\delta:E\subset\cup_j(z_j+B_{r_j}),z_j\in\R^2\big\}$. The Hausdorff content is an outer measure that is always finite for bounded sets and smaller than the Hausdorff measure. It is worth noting that $\HH_\delta(E)=0$ if and only if the $\delta$-dimensional Hausdorff measure of $E$ vanishes. 

\begin{theorem}\label{rough-u}
Let $\delta>0$ and $\omega\subset B_1\cap l_0$ satisfy $\HH_\delta(\omega)>0$ for some line $l_0$ in $\R^2$ with the normal vector $\ee_0$. There exist some constants $C$ and $\alpha>0$ depending only on $\Lambda$, $\delta$ and $\HH_\delta(\omega)$ such that for any solution $u$ of \eqref{eq1} subject to \eqref{elliptic} with $A\nabla u\cdot\ee_0=0$ on $B_1\cap l_0$, we have 
\begin{align*}
\sup\nolimits_{B_1} |u| \le C \sup\nolimits_{\omega}|u|^\alpha \sup\nolimits_{B_2}|u|^{1-\alpha}. 
\end{align*}
\end{theorem}
	
\begin{theorem}\label{gradient}
Let $\delta>0$ and $\Omega\subset B_1$ satisfy $\HH_\delta(\Omega)>0$. There exist some constants $C$ and $\alpha>0$ depending only on $\Lambda$, $\delta$, $\gamma$ and $\HH_\delta(\Omega)$ such that for any solution $u$ of \eqref{eq1} subject to \eqref{elliptic} and \eqref{Holder}, we have 
\begin{align*}
\sup\nolimits_{B_1} |\nabla u| \le C \sup\nolimits_{\Omega}|\nabla u|^\alpha \sup\nolimits_{B_2}|\nabla u|^{1-\alpha} . 
\end{align*}
The same estimate holds for any $u$ solving \eqref{eqnon} subject to \eqref{elliptic} only. 
\end{theorem}

The above results are related to the unique continuation property for two-dimensional elliptic equations with bounded measurable coefficients. The properties of equations in nondivergence and divergence forms were proved in \cite{BN} (see also \cite{BJS}) and \cite{AM}, respectively. By contrast, it has been known since \cite{Plis} and \cite{Miller} that if the coefficients are H\"older continuous, then one is able to construct a nontrivial solution, which vanishes on an open subset, to elliptic equations in either non-divergence or divergence forms with dimensions greater than three. 
	
The case in the plane is special owing to the theory of quasiregular mappings and the representation theorem; see \cite{BN}, \cite{Bo} and \cite{AIM}. It reduces the analysis from solutions (or gradients) of elliptic equations to holomorphic functions. The main idea of the proof of propagation of smallness for holomorphic functions (see Proposition~\ref{holo} below) is based on the complex analysis arguments used in \cite{Kov} and \cite{AE}. Two basic observations are that the ratio of a holomorphic function $F$ and the polynomial sharing the same zeros as $F$ is holomorphic and non-vanishing, and the logarithm of modulus of a non-vanishing holomorphic function is harmonic. The properties of harmonic functions and the Remez-type inequality for holomorphic polynomials (in the version obtained in \cite{Bru1} and \cite{FY}) then allow us to derive propagation of smallness from wild sets. We also point out that once their gradients are well-defined, the number of critical points of solutions can be computed by counting the zeros of holomorphic functions, which allows for quantitative estimates of the size of the critical set. For the smooth coefficients case, the computation can be found in \cite{Ale88} or \cite{ZhuJ}, while discussions on the general case are presented in \cite{AM}. 

Let us briefly discuss the higher-dimensional case. Consider the solution $u$ to the elliptic equation $\nabla\cdot(A\nabla u)=0$ with Lipschitz coefficients in some domain of $\R^d$ for $d\ge3$. It was proved in \cite{LM18} that the propagation of smallness for $u$ holds from sets of positive $(d-1+\delta)$-Hausdorff content for any $\delta>0$. The result is sharp in the sense that zeros of harmonic functions in $\R^d$ may have positive $(d-1)$-Hausdorff content. The propagation of smallness for $|\nabla u|$ from sets of positive $(d-1-\epsilon)$-Hausdorff content for some (small) $\epsilon>0$ was also obtained in \cite{LM18}. As the zeros of $|\nabla u|$ was shown in \cite{NV} to have finite $(d-2)$-Hausdorff measure, it was conjectured in \cite{LM18} that the result for $|\nabla u|$ should be expected to hold from sets of positive $(d-2+\delta)$-Hausdorff content for any $\delta>0$. Some partial results in the analytic setting can be found in \cite{Malinnikova}. One may refer to the review \cite[Sections~7\&8]{LM-review} for further discussion in this direction. 

In the two-dimensional setting, Theorem~\ref{gradient} provides propagation of smallness for gradients of solutions to both divergence form equations with H\"older continuous coefficients and nondivergence form equations with measurable coefficients, from sets of positive $\delta$-dimensional Hausdorff content for any $\delta>0$. We note that the result of Theorem~\ref{gradient} for divergence form equations with Lipschitz coefficients has been obtained in the first version of this paper and in \cite{Foster}, which relied on the findings in \cite{Ale88} or the utilization of isothermal coordinates.. 

The outline of the article is as follows. We prove Theorem~\ref{rough-u} in Section~\ref{sec-rough} and Theorem~\ref{gradient} in Section~\ref{sec-gradient}. Some remarks on applications of Theorem~\ref{rough-u} to spectral inequalities and null controllability of heat equations with rough coefficients are presented in Appendix~\ref{sec-control}. 
	
In the rest of the article, the constant $C>0$, which will appear in the proofs below and depend only on $\Lambda$, $\delta$, $\gamma$ and $\HH_\delta(\omega)$ (or $\HH_\delta(\Omega)$), may be changed line by line. 

\section{Propagation of smallness with rough coefficients}\label{sec-rough}
This section is devoted to the proof of Theorem~\ref{rough-u}. We first review the notion of \emph{$A$-harmonic conjugate} (or sometimes called \emph{stream function}); one may refer to \cite{Rene} and \cite{AIM}. Consider the solution $u(z)$ of \eqref{eq1} subject to \eqref{elliptic}, and identify $z=x_1+ix_2$ for $x_1,x_2\in\R$. The $A$-harmonic conjugate $v$ of $u$ is defined by 
\begin{align}\label{vv}
\nabla v=\left(\begin{smallmatrix}0 & -1 \\1 & 0 \end{smallmatrix}\right) A\nabla u
\end{align}
so that $v$ verifies the following elliptic equation, 
\begin{align*}
\nabla\cdot\left(\det(A)^{-1}A\nabla v\right)
=\nabla\cdot\left(\left(\begin{smallmatrix}0 & -1 \\1 & 0 \end{smallmatrix}\right)\nabla u\right)=0,  
\end{align*}
where the fact that $\det(A)^{-1}A=\det(A)^{-1}A^T=-\left(\begin{smallmatrix}0 & -1 \\1 & 0 \end{smallmatrix}\right)A^{-1}\left(\begin{smallmatrix}0 & -1 \\1 & 0 \end{smallmatrix}\right)$ was used. Now that $v\in H^1_{loc}(B_4)$ is unique up to an additive constant, we may additionally assume that $v(z_0)=0$ for some $z_0\in B_1$. Define the function $f:B_4\rightarrow\C$ by 
\begin{align*}
f(z):=u(z)+iv(z). 
\end{align*}
By definition and \eqref{elliptic}, we have 
\begin{align*}
\begin{split}	
|Df|^2&= |\nabla u|^2+|A\nabla u|^2 
\le (\Lambda+\Lambda^{-1})A\nabla u\cdot\nabla u\\
&=(\Lambda+\Lambda^{-1})\left(\begin{smallmatrix}0 & 1 \\-1 & 0 \end{smallmatrix}\right)\nabla v\cdot\nabla u
= (\Lambda+\Lambda^{-1})Jf, 
\end{split}
\end{align*}
for the norm $|Df|^2:=|\nabla u|^2+|\nabla v|^2$ and the Jacobian $Jf:=\partial_{x_1}u\,\partial_{x_2}v-\partial_{x_2}u\,\partial_{x_1}v$. We are now positioned to recall the concept of quasiregular mapping. 

\begin{definition}
Let $U$ be an open set in $\C$ and $K\ge1$ be a constant. A complex-valued function $f\in H^1_{loc}(U)$ satisfying $|Df|^2\le (K+K^{-1})Jf$ almost everywhere in $U$ is said to be a $K$-\emph{quasiregular mapping} on $U$. 
\end{definition}
%In particular, the Jacobian of any quasiregular mapping is nonnegative so that the mapping is orientation-preserving. 

The following representation theorem plays a pivotal role in two-dimensional elliptic theory, which was first obtained in \cite[\S2~Representation~Theorem]{BN} and \cite[Theorem~4.4]{Bo}. For our intended applications, we refer to \cite[II 2.1]{Rene} or \cite[Corollary 5.5.3]{AIM}, and formulate it as follows, which is adequate for our needs. 

\begin{lemma}\label{representation}
Let the constant $K\ge1$. Any $K$-quasiregular mapping $f:B_4\rightarrow\C$ can be written as
\begin{align*}
f(z)=F\circ\chi(z).  
\end{align*}
Here $F$ is holomorphic in $B_4$, and $\chi:B_4\rightarrow B_4$, with $\chi(z_0)=0$, is a $K$-quasiconformal homeomorphism satisfying
\begin{align}\label{chi}
M^{-1}\left|z-z'\right|^{1/\beta} \le \left|\chi(z)-\chi(z')\right| \le M \left|z-z'\right|^{\beta} {\quad\rm for\ any\ }z,z'\in B_4, 
\end{align}
where the exponent $\beta=1/K\in(0,1]$ and the constant $M>1$ depends only on $K$. 
\end{lemma}

Let us turn to the main result of this section. It gives the corresponding quantitative propagation of smallness for holomorphic functions. 
\begin{proposition}\label{holo}
Let $\theta>0$, $R\in(0,1/4)$ and $G\subset B_R$ satisfy $\HH_\theta(G)>0$. There exists some constant $C>0$ depending only on $\theta$ and $\HH_\theta(G)$ such that for any holomorphic function $F$ in $B_{5R}$, we have 
\begin{align*}
\sup\nolimits_{B_R}|F|
\le \sup\nolimits_{G}|F|^\frac{1}{1+C} \sup\nolimits_{B_{4R}}|F|^\frac{C}{1+C}. 
\end{align*} 
\end{proposition}

\begin{proof}
Regarding to the holomorphic function $f(\zeta)$ for $\zeta\in B_4$, we assume that $N\in\N$ and $\zeta_1,\ldots,\zeta_N$ are the zeros of $F$ in $B_{2R}$, listed with multiplicities, and $\zeta_0\in \overline{B_R}$ satisfies $|F(\zeta_0)|=\sup_{B_R}|F(\zeta)|$. Consider the polynomial $P(\zeta)$ sharing the same zeros as $F(\zeta)$ in $B_{2R}$ so that their ratio $h(\zeta)$ is  holomorphic and non-vanishing in $B_{2R}$; more precisely, 
\begin{align*}
P(\zeta):=\prod\nolimits_{k=1}^N (\zeta-\zeta_k),\quad h(\zeta):=\frac{F(\zeta)}{P(\zeta)}. 
\end{align*}
It follows that $\log|h(\zeta)|$ is harmonic in $B_{2R}$, and 
\begin{align}\label{h0}
|h(\zeta_0)| \ge \sup\nolimits_{B_{R}}|F|. 
\end{align}
Moreover, by the maximum modulus principle, we have 
\begin{align}\label{hFP}
\begin{aligned}
\sup\nolimits_{B_{4R}}|h|
\le \sup\nolimits_{\partial B_{4R}}|F|\sup\nolimits_{\partial B_{4R}}|P^{-1}|
\le C^N \sup\nolimits_{B_{4R}}|F|. 
\end{aligned}
\end{align}
Applying the Harnack inequality to $\sup_{B_{4R}}\log|h|-\log|h(\zeta)|$ yields that 
\begin{align*}
\sup\nolimits_{B_{4R}}\log|h|-\inf\nolimits_{B_R}\log|h| 
\le C\sup\nolimits_{B_{4R}}\log|h|-C\log|h(\zeta_0)|, 
\end{align*}
which is equivalent to 
\begin{align*}
|h(\zeta_0)|^C\sup\nolimits_{B_{4R}}|h| \le \sup\nolimits_{B_{4R}}|h|^C \inf\nolimits_{B_R}|h|.  
\end{align*}
Combining this with \eqref{h0} and \eqref{hFP} implies that 
\begin{align}\label{FFh}
\sup\nolimits_{B_R}|F|^C\sup\nolimits_{B_{4R}}|h| 
\le C^N \sup\nolimits_{B_{4R}}|F|^C \inf\nolimits_{B_R}|h|. 
\end{align}
Recall the Remez-type inequality (see \cite[Theorem~4.3]{FY}) for holomorphic polynomials of degree $N$ that 
\begin{align}\label{Remez}
\sup\nolimits_{B_R}|P|\le \left(20/\HH_{\theta}(G)\right)^{N/\theta} \sup\nolimits_{G}|P|. 
\end{align}
Multiplying the inequalities \eqref{FFh} and \eqref{Remez}, we deduce 
\begin{align*}
\sup\nolimits_{B_R}|F|^{1+C}
&\le C^N \sup\nolimits_{B_{4R}}|F|^C \inf\nolimits_{B_R}|h|\, \sup\nolimits_{G}|P|\\
&\le C^N \sup\nolimits_{B_{4R}}|F|^C \sup\nolimits_{G}|F|. 
\end{align*}
Since Jensen's formula shows that the number of zeros of $F$ in $B_{2R}$ satisfies $$N\le C\sup\nolimits_{B_{4R}}\log\left(|F|/|F(\zeta_0)|\right),$$ 
we derive the desired result. 
\end{proof}

\begin{corollary}\label{rough}
Let $\delta>0$ and $\Omega\subset B_1$ satisfy $\HH_\delta(\Omega)>0$. There exist some constants $C,\alpha>0$ depending only on $\Lambda$, $\delta$ and $\HH_\delta(\Omega)$ such that for any solution $u$ of \eqref{eq1} subject to \eqref{elliptic} with its $A$-harmonic conjugate $v$ satisfying $v(z_0)=0$ for some $z_0\in B_1$, we have 
\begin{align*}
\sup\nolimits_{B_1} |u| \le C \sup\nolimits_{\Omega}|u+iv|^\alpha \sup\nolimits_{B_2}|u|^{1-\alpha}. 
\end{align*}
\end{corollary}
	
\begin{proof}
The derivation is reduced to the analysis of holomorphic functions with the aid of Lemma~\ref{representation}. With the complex functions $F$ and $\chi$ provided in Lemma~\ref{representation}, for $z=\chi^{-1}(\zeta)\in B_4$, we have 
\begin{align*}
F(\zeta)=f\circ\chi^{-1}(\zeta)=u\circ\chi^{-1}(\zeta)+iv\circ\chi^{-1}(\zeta). 
\end{align*} 
Without loss of generality, by the property \eqref{chi} of $\chi$, we may assume that $\chi(B_2)\subset B_2$, and $\chi(B_r),\chi(\Omega)\subset B_R$, where the constants $R,r\in(0,1/4)$ depend only on $\Lambda$. Furthermore, it follows from \eqref{chi} that there is some $C_*>0$ depending only on $\Lambda$ and $\delta$ such that 
\begin{align*}
\HH_\delta(\Omega)\le C_*\HH_{\delta\beta}(\chi(\Omega)). 
\end{align*}	
Applying Proposition~\ref{holo} with $\theta=\delta\beta$ and $G=\chi(\Omega)$ yields that 
\begin{align*}
\sup\nolimits_{B_r}|f|
\le C\sup\nolimits_{\Omega}|f|^\frac{1}{1+C}
\sup\nolimits_{\chi(B_2)}\big(|u\circ\chi^{-1}|+|v\circ\chi^{-1}|\big)^\frac{C}{1+C}. 
\end{align*} 
Since $u\circ\chi^{-1}(\zeta)+iv\circ\chi^{-1}(\zeta)$ is holomorphic and $v\circ\chi^{-1}=0$ at the point $\chi(z_0)\in\chi(B_1)$, we obtain from the Cauchy–Riemann equations and the gradient estimate for the harmonic function $u\circ\chi^{-1}(\zeta)$ that
\begin{align*}
\begin{aligned}
\sup\nolimits_{\chi(B_1)}|v\circ\chi^{-1}|
&\le C\sup\nolimits_{\chi(B_1)}|\nabla_\zeta (u\circ\chi^{-1})|
\le C \sup\nolimits_{\chi(B_2)}|u\circ\chi^{-1}|. 
\end{aligned}
\end{align*} 
Gathering the above two estimates, we obtain  
\begin{align*}
\sup\nolimits_{B_r} |u+iv| \le C \sup\nolimits_{\Omega}|u+iv|^{\frac{1}{1+C}} \sup\nolimits_{B_2}|u|^{\frac{C}{1+C}}. 
\end{align*}
We then conclude the desired result by a covering argument. 
\end{proof}
	
Theorem~\ref{rough-u} is then a direct consequence of the above result. 
	
\begin{proof}[Proof of Theorem~\ref{rough-u}]
We may assume that the $A$-harmonic conjugate $v$ of $u$ satisfies $v(z_0)=0$ for some $z_0\in B_1\cap l_0$. Since 
\begin{align*}
\nabla v\cdot\ee_0^\perp=\left(\begin{smallmatrix}0 & -1 \\1 & 0 \end{smallmatrix}\right) A\nabla u\cdot\ee_0^\perp=0 {\quad\rm on\ }B_1\cap l_0, 
\end{align*}
we deduce that $v=0$ on $B_1\cap l_0$. Corollary~\ref{rough} then implies the result as claimed. 
\end{proof}

\section{Propagation of smallness for gradients}\label{sec-gradient}
This section is devoted to the proof of Theorem~\ref{gradient}. In general, the supremum of $|\nabla u|$ for $u$ solving \eqref{eq1} subject to \eqref{elliptic} does not make sense, especially over sets of Hausdorff dimension less than one. We thus have to strengthen the regularity assumption on the coefficients of \eqref{eq1}. In particular, if the leading coefficients is H\"older continuous, then the classical Schauder theory says that the gradients of the solutions to elliptic equations in divergence form are H\"older continuous. In order to establish the propagation of smallness for gradients of $u$ solving \eqref{eq1} subject to \eqref{elliptic} and \eqref{Holder}, we will use a perturbation argument inspired by the Schauder theory. The proof presented in Subsection~\ref{divergence} below is essentially based on Proposition~\ref{holo}, the propagation of smallness for holomorphic functions. 
	
As for a solution of \eqref{eqnon} subject to \eqref{elliptic}, it is straightforward to check that its gradient forms a quasiregular mapping (see Subsection~\ref{nondivergence} below); one may also refer to \cite[\S 12.2]{GT} and \cite[\S~II.1]{Rene}. It then follows from the same analysis of holomorphic functions as in the previous section that Theorem~\ref{gradient} holds for solutions to \eqref{eqnon} with rough coefficients.

\subsection{Elliptic equations in divergence form}\label{divergence}
We are in a position to present the proof of Theorem~\ref{gradient} for the solution $u$ of \eqref{eq1} subject to \eqref{elliptic} and \eqref{Holder}. 

Recall that $v$ is the $A$-harmonic conjugate of $u$ and $f=u+iv$. The functions $u$ and $v$ solve the elliptic equations \eqref{eq1} and \eqref{vv}, respectively. 
Without loss of generality, we may assume that 
\begin{align*}
\sup\nolimits_{B_2}|\nabla f|=1. 
\end{align*} 
In the light of the H\"older condition \eqref{Holder}, the Schauder estimate (see for instance \cite[Corollary~6.3]{GT}) implies that 
\begin{align}\label{Schauder}
\|\nabla f\|_{C^\gamma(B_{3/2})}\le C. 
\end{align}
Let $\eta\in(0,1/4)$ be a constant to be determined. We may pick $N_\eta$ balls of radius $\eta$ to cover $B_1$; for instance, $N_\eta=4/\eta^2$. Then there is a ball $B_\eta(z_*)$ centred at some $z_*\in B_1$ such that
\begin{align}\label{measure-B}
\HH_\delta(B_\eta(z_*)\cap\Omega)\ge \HH_\delta(\Omega)/N_\eta=\eta^2\HH_\delta(\Omega)/4. 
\end{align}
By rotating and dilating the coordinates, we can assume that $A(z_*)=I_{2\times 2}$. Consequently, by the H\"older condition \eqref{Holder}, we have
\begin{align}\label{elliptic-pert}
(1+C_\Lambda\eta^\gamma)^{-1}|\xi|^2 \le A\xi\cdot\xi\le (1+C_\Lambda\eta^\gamma)|\xi|^2 {\quad\rm for\ any\ }\xi\in\R^2,\ z\in B_{2\eta}(z_*), 
\end{align}
for some constant $C_\Lambda>1$ depending only on $\Lambda$. We are going to derive local estimates in $B_{2\eta}(z_*)$. Let $\hat{z}\in B_\eta(z_*)\cap\Omega$ be fixed and let $\varepsilon\in(0,\eta]$ so that $B_\varepsilon(\hat{z})\subset B_{2\eta}(z_*)$. By \eqref{Schauder}, we have 
\begin{align*}
\begin{aligned}
\sup\nolimits_{B_\varepsilon(\hat{z})}|f-f(\hat{z})|
\le \varepsilon\sup\nolimits_{B_\varepsilon(\hat{z})}|\nabla f|
\le C\varepsilon|\nabla f(\hat{z})|+C\varepsilon^{1+\gamma}. 
\end{aligned}
\end{align*}
Applying Lemma~\ref{representation} in $B_{2\eta}(z_*)$, we have the representation formula $f=F\circ\chi$ for a holomorphic function $F$ and a quasiconformal homeomorphism $\chi$ such that $\chi(z_*)=z_*$. In particular, the application with the elliptic condition \eqref{elliptic-pert} yields a constant $\beta:=(1+C_\Lambda\eta^\gamma)^{-1}$ ensuring the property \eqref{chi}. Hence, for any $\varepsilon\in(0,\eta]$, 
\begin{align*}
\sup\nolimits_{B_{r_\varepsilon}(\chi(\hat{z}))}|F-F\circ\chi(\hat{z})|
\le C\sup\nolimits_{B_\varepsilon(\hat{z})}|f-f(\hat{z})|, 
\end{align*}
where we set $r_\varepsilon:=c\varepsilon^{1/\beta}$ for some constant $c\in(0,1)$ depending only on $\Lambda$. 
By the gradient estimate for the holomorphic function, we have
\begin{align*}
|\nabla F(\chi(\hat{z}))|
\le Cr_\varepsilon^{-1}\sup\nolimits_{B_{r_\varepsilon}(\chi(\hat{z}))}|F-F\circ\chi(\hat{z})|. 
\end{align*}
Gathering the above three estimates, we obtain 
\begin{align*}
|\nabla F(\chi(\hat{z}))|
\le C\varepsilon^{1-1/\beta}|\nabla f(\hat{z})| +C\varepsilon^{1+\gamma-1/\beta}. 
\end{align*}
Since $1+\gamma/2-1/\beta=\gamma/2-C_\Lambda\eta^\gamma$, we choose $\eta\in(0,1/4)$ such that $C_\Lambda\eta^\gamma=\gamma/2$ which implies that $1-1/\beta=-\gamma/2$, and $\eta$ depends only on $\Lambda$ and $\gamma$. Then, taking $\varepsilon:=\min\{\eta,\sup\nolimits_\Omega|\nabla f|\}$, we deduce that for $\gamma':=\min\{1-\gamma/2,\gamma/2\}$,  
\begin{align*}
\begin{aligned}
|\nabla F(\chi(\hat{z}))|
\le C\varepsilon^{-\gamma/2}|\nabla f(\hat{z})|+C\varepsilon^{\gamma/2}
\le C\sup\nolimits_\Omega|\nabla f|^{\gamma'}. 
\end{aligned}
\end{align*}
Now that $\hat{z}\in B_\eta(z_*)\cap\Omega$ is arbitrary and $\chi(z_*)=z_*$, armed with \eqref{measure-B}, applying Proposition~\ref{holo} to $\nabla F$ with $G=\chi(B_\eta(z_*)\cap\Omega)$ implies that, for some $R,C>0$ depending only on $\Lambda$ and $\eta$, 
\begin{align*}
\begin{aligned}
\sup\nolimits_{B_R(z_*)}|\nabla F|
\le \sup\nolimits_{\chi(B_\eta(z_*)\cap\Omega)}|\nabla F|^\frac{1}{1+C}
\le C\sup\nolimits_{\Omega}|\nabla f|^\frac{\gamma'}{1+C}. 
\end{aligned}
\end{align*}
In view of the gradient estimate for $f$ and \eqref{chi} with $\chi(z_*)=z_*$, we have 
\begin{align*}
\begin{aligned}
\sup\nolimits_{B_r(z_*)}|\nabla f|
&\le C \sup\nolimits_{B_{2r}(z_*)}|f-f(z_*)|\\
&\le C \sup\nolimits_{B_R(z_*)}|F-F(z_*)| 
\le C \sup\nolimits_{B_R(z_*)}|\nabla F|, 
\end{aligned}
\end{align*}
where the constant $r>0$ depends only on $\Lambda$ and $\gamma$. We thus conclude that  
\begin{align*}
\sup\nolimits_{B_r(z_*)}|\nabla f|
\le C\sup\nolimits_{\Omega}|\nabla f|^\frac{\gamma'}{1+C}. 
\end{align*}
Rescaling back and recalling $f=u+iv$, as well as the definition of $v$, we arrive at 
\begin{align*}
\sup\nolimits_{B_r(z_*)}|\nabla u|
&\le C\sup\nolimits_{\Omega}(|\nabla u|+|\nabla v|)^\frac{\gamma'}{1+C}\sup\nolimits_{B_2}(|\nabla u|+|\nabla v|)^\frac{1+C-\gamma'}{1+C}\\
&\le C\sup\nolimits_{\Omega}|\nabla u|^\frac{\gamma'}{1+C}\sup\nolimits_{B_2}|\nabla u|^\frac{1+C-\gamma'}{1+C}. 
\end{align*}
The desired result then follows from a covering argument.

\subsection{Elliptic equations in nondivergence form}\label{nondivergence}
Let us start by reviewing some basic facts from \cite[\S 12.2]{GT} (see also \cite[\S~II.1]{Rene}). Consider the solution $u(z)$ of \eqref{eqnon} subject to \eqref{elliptic} with $z=x_1+ix_2$ for $x_1,x_2\in\R$. Define 
\begin{equation*}\label{pq}
p(z):=\partial_{x_1}u(z),\quad q(z):=\partial_{x_2}u(z),\quad g(z):=q(z)+ip(z). 
\end{equation*}
Since $\partial_{x_1}q=\partial_{x_2}p$, multiplying \eqref{eqnon} by $\partial_{x_2}q$ and $\partial_{x_1}p$ yields that 
\begin{align*}
a_{11}(\partial_{x_1}q)^2 +2a_{12}\partial_{x_1}q\,\partial_{x_2}q + a_{22}(\partial_{x_2}q)^2 
&= a_{11}(\partial_{x_1}q\,\partial_{x_2}p - \partial_{x_2}q\,\partial_{x_1}p),\\
a_{11}(\partial_{x_1}p)^2 +2a_{12}\partial_{x_1}p\,\partial_{x_2}p + a_{22}(\partial_{x_2}p)^2 
&= a_{22}(\partial_{x_1}q\,\partial_{x_2}p - \partial_{x_2}q\,\partial_{x_1}p). 
\end{align*}
Due to \eqref{elliptic}, we have $|\partial_{x_1}g|^2+|\partial_{x_2}g|^2\le (1+\Lambda^2) Jg$ for the Jacobian $Jg:=\partial_{x_1}q\,\partial_{x_2}p-\partial_{x_2}q\,\partial_{x_1}p$. It turns out that $g\in H^1_{loc}(B_4)$ is quasiregular. By the representation theorem (Lemma~\ref{representation}), we have \begin{align*}
g(z)=F\circ\chi(z),
\end{align*}
where $F$ is holomorphic in $B_4$, and $\chi:B_4\rightarrow B_4$ is a homeomorphism satisfying $\chi(0)=0$ and \eqref{chi}. The argument presented in Section~\ref{sec-rough} can be thus applied in this setting. Indeed, in view of Proposition~\ref{holo}, we have    
\begin{align*}
\sup\nolimits_{B_R}|F|
\le \sup\nolimits_{\chi(\Omega)}|F|^\frac{1}{1+C} \sup\nolimits_{B_{4R}}|F|^\frac{C}{1+C}. 
\end{align*} 
Now that $F=(\partial_{x_1}u)\circ\chi^{-1}+i(\partial_{x_2}u)\circ\chi^{-1}$, we are able to conclude Theorem~\ref{gradient} for solutions of \eqref{eqnon} subject to \eqref{elliptic}.

\appendix
\section{Spectral inequality and null controllability}\label{sec-control}
We briefly discuss the application of propagation of smallness, Theorem ~\ref{rough-u}, to spectral inequalities and null controllability of heat equations with bounded measurable coefficients. 

Let $\T$ be the periodic unit interval, and the function $a:\T\rightarrow\R$ be measurable and satisfy $\Lambda^{-1}\le a(x)\le\Lambda$ in $\T$ for some constant $\Lambda>1$. Consider the one-dimensional eigenvalue problem 
\begin{align*}
-\partial_x \left(a(x)\,\partial_x e_k(x)\right)=\lambda_ke_k(x) {\quad\rm in\ }\T. 
\end{align*} 
Then the family of eigenfunctions $\{e_k(x)\}_{k\in\N}$ forms an orthonormal basis of $L^2(\T)$, and the family of eigenvalues $\{\lambda_k\}_{k\in\N}$ satisfies $\lambda_k\ge0$ for any $k\in\N$ and $\lambda_k\rightarrow\infty$ as $k\rightarrow\infty$. Denote by $\Pi_\lambda$ the orthogonal projection onto the space spanned by $\{e_k:\lambda_k\le\lambda\}$. We have the following spectral inequality. 
	
\begin{proposition}\label{spectral}
Let $\delta>0$ and $\omega\subset\T$ satisfy $\HH_{\delta}(\omega)>0$. There exists some constant $C>0$ depending only on $\Lambda$, $\delta$ and $\HH_{\delta}(\omega)$ such that for any $\phi\in L^2(\T)$ and any $\lambda\ge1$, we have 
\begin{align*}
\sup\nolimits_{\T} |\Pi_\lambda\phi| \le e^{C\sqrt{\lambda}} \sup\nolimits_{\omega}|\Pi_\lambda\phi|. 
\end{align*}
\end{proposition}
	
\begin{proof}
We may write $\Pi_\lambda\phi(x)=\sum\nolimits_{\lambda_k\le\lambda} \phi_ke_k(x)$ for $\phi_k\in\R$. The function
\begin{align*}
u(x,y):=\sum\nolimits_{\lambda_k\le\lambda} \phi_ke_k(x) \cosh(\sqrt{\lambda_k}y), \quad (x,y)\in\T\times(-4,4),
\end{align*} 
satisfies $\partial_yu(x,0)=0$ and $u(x,0)=\Pi_\lambda\phi(x)$ for $x\in\T$, and 
\begin{align*}
\partial_x\left(a(x)\,\partial_x u\right) +\partial_y^2u=0 {\quad\rm in\ }\T\times(-4,4).  
\end{align*}
Applying Theorem~\ref{rough-u} to $u$ yields that for some constants $C,\alpha>0$, 
\begin{align*}
\begin{aligned}
\sup\nolimits_\T|\Pi_\lambda\phi| 
\le \sup\nolimits_{\T\times(-1,1)}|u| 
\le C \sup\nolimits_{\omega}|\Pi_\lambda\phi|^\alpha \sup\nolimits_{\T\times(-2,2)}|u|^{1-\alpha}. 
\end{aligned}
\end{align*}
By the Sobolev inequality and the fact that $\|\partial_xe_k\|_{L^2(\T)}^2\le\Lambda\lambda_k\|e_k\|_{L^2(\T)}^2$, we have 
\begin{align*}
\begin{aligned}
\sup \nolimits_{\T\times(-2,2)}|u|
\le  C\|u\|_{L^\infty_y((-2,2),H_x^1(\T))}
\le  e^{C\sqrt{\lambda}}\|\Pi_\lambda\phi\|_{L^2(\T)}. 
\end{aligned}
\end{align*}
We then conclude the proof by gathering the above two estimates. 
\end{proof}
	
The problem of null controllability of multi-dimensional heat equations with Lipschitz coefficients from open control sets has been intensively developed since \cite{Imanuvilov} and \cite{LR}. The null controllability of one-dimensional heat equations with rough coefficients from open sets was proved in \cite{AE}, and the result from sets of positive Lebesgue measure was given in \cite{ApraizE}. Proposition~\ref{spectral} would imply the result from the control set $\omega$ satisfying $\HH_{\delta}(\omega)>0$ for any fixed $\delta>0$.  
	
\begin{proposition}\label{control}
Let $T>0$, $\delta>0$, and $\omega$ be a closed subset of $\T$ with $\HH_{\delta}(\omega)>0$. For any $w_0\in L^2(\T)$, there exists a Borel measure $m(t,x)$ supported in $(0,T)\times\T$ such that the solution $w(t,x)$ to 
\begin{align*}\label{control-heat}
\partial_tw(t,x) = \partial_x\left(a(x)\,\partial_x w(t,x)\right)  + m(t,x)\mathds{1}_\omega {\quad\rm in\ }(0,T)\times\T, 
\end{align*}
associated with the initial data $w(0,\cdot)=w_0$ in $\T$, satisfies $w(T,\cdot)=0$ in $\T$. 
\end{proposition}
	
The proof of the above result consists in the spectral inequality (Proposition~\ref{spectral}), the decay property of the semigroup $e^{t\partial_x(a(x)\partial_x\cdot)}$, and the duality argument (see for instance \cite[Section~5]{BurqM}). Since it is now quite standard to combine these ingredients, we omit the proof; one may refer to \cite{BurqM} for details. One is also able to generalize the null controllability result to more general one-dimensional heat equations (with lower order terms) associated with certain boundary conditions from space-time control sets; see \cite{AE} and \cite{BurqM}.

\end{document}